\documentclass[a4paper,10pt]{amsart}

\usepackage{amssymb}
\usepackage{amsthm}
\usepackage{graphicx}
\newcommand{\R}{\mathbb{R}}
\newcommand{\range}{\operatorname{range}}

\renewcommand{\u}{\mathsf{u}}
\newtheorem{lemma}{Lemma}
\newtheorem{theorem}{Theorem}

\newtheorem{proposition}{Proposition}

\title[Errors of spectral enhancement]{Error bounds for spectral enhancement
which are based on variable Hilbert scale inequalities}

\author{Markus Hegland}
\date{\today}

\begin{document}

\maketitle

\begin{abstract}
  Spectral enhancement -- which aims to undo spectral broadening -- leads to 
integral equations which are ill-posed and require special regularisation
techniques for their solution. Even when an optimal regularisation technique is
used, however, the errors in the solution -- which originate in data
approximation errors -- can be substantial and it is important to have good
bounds for these errors in order to select appropriate enhancement methods. A
discussion of the causes and nature of broadening provides regularity or source
conditions which are required to obtain bounds for the regularised solution of
the spectral enhancement problem. The source conditions do only in special cases
satisfy the requirements of the standard convergence theory for ill-posed
problems. Instead we have to use variable Hilbert scales and their interpolation
inequalities to get error bounds. The error bounds in this case turn out to be
of the form $O(\epsilon^{1-\eta(\epsilon)})$ where $\epsilon$ is the data error
and $\eta(\epsilon)$ is a function which tends to zero when $\epsilon$ tends to
zero. The approach is demonstrated with the Eddington correction formula and
applied to a new spectral reconstruction technique for Voigt spectra. In this
case $\eta(\epsilon)=O(1/\sqrt{|\log\epsilon|})$ is found.
\end{abstract}

\section{Introduction}

One of the computational challenges in spectroscopy is the separation of
overlapping spectral lines. This separation can be achieved by computationally
narrowing the spectral lines and thus enhancing the resolution or correcting the
spectrum. The class of methods of resolution enhancement considered here is
based on the solution of linear Fredholm integral equations of the first kind
using observed data for the right hand side. The basic approach was first
analysed in~\cite{AllGG64} but it goes back in principle to work by
Stokes~\cite{Sto48}. The effect of data errors has to be analysed carefully,
especially since the enhancement problem is ill-posed. This analysis is
performed in the following using variable Hilbert scales~\cite{Heg92,Heg95}. A
more traditional error analysis which can be found in~\cite{Gro84} is not
directly applicable here as the source conditions are non-standard. However, in
contrast to many other ill-posed problems, here the underlying physical model
does suggest specific source conditions. If $f$ is the enhanced spectrum and
$f_\alpha$ an (optimal order) regularised approximation of $f$ then  bounds
of the form 
$$ \|f-f_\alpha\| \leq \epsilon^{1-\eta(\epsilon)} $$ 
are found where $\epsilon$ is the residual of $f_\alpha$. In the classical case
the $\eta(\epsilon)$ is constant, in contrast it is shown here that this
exponent slowly decreases to zero with $\epsilon\rightarrow 0$. 

A new enhancement method based on Lorentz kernels for Voigt spectra is shown to
provide good performance compared to more traditional methods like the Eddington
correction as it capitalises more on the smoothness of the data and does not
require any advanced knowledge of the proportions of the Gaussian and Lorentzian
components in the Voigt spectrum. If a spectrum contains a Gaussian component
the error bound is of order $O(\epsilon^{1-c/\sqrt{|\log\epsilon|}})$ and the
convergence rate thus grows with $\epsilon\rightarrow 0$. For very small
$\epsilon$ one can find very close to $O(\epsilon)$ convergence, however, this
depends on the level of enhancement required. Experiments show that this method
leads to a reduction of linewidth of more than a factor of two in the case of a
5\% data error.

In the remaining parts of this section a brief review of broadening 
mechanisms are given, in addition to a short discussion of a least squares method to
determine the location and strength of spectral lines. In section~2 we present the
integral equation framework for resolution enhancement and illustrate this
with the Eddington correction formula and Stokes correction by partial
Gaussian deconvolution. In section~3 the method using Lorentz deconvolution
for Gaussian and Voigt spectra is discussed in terms of the errors. Section~4
then provides some demonstrations of the enhancement properties of this
Lorentz deconvolution which in particular illustrates the broadening
effects of noise and regularisation. In the concluding section~5 related and 
open problems are considered. 

\subsection{Models of spectra and broadening}

In the natural sciences, a spectrum is a distribution of photon counts over
energy or frequency. Since Fraunhofer's work in 1814 it is well known that this
distribution is concentrated along lines, both for emission and absorption
spectra. The existence of these \emph{spectral lines} was later confirmed by
quantum mechanics. Their importance is due to the fact that they provide
information about the energy levels of the electrons and thus insights into the
structure and composition of the originating substrate. Spectroscopy has been
for a long time one of the most important tools in experimental science. A
simple model for a spectrum based on the Fraunhofer spectral lines would consist
of a probability measure with discrete support.

Almost simultaneously with Fraunhofer's discovery it was realised that spectral
lines have a non-zero width. This broadening originates from many different
physical effects and a discussion of spectral broadening can be found in a
variety of different books and journals, see for
example~\cite{Bre81,BowM89,SobVY95,Lin02,Bra03,Sim03,Bal06,Dra06,Kha06,Irw07}.
In order to get a basic idea we review some of the most important mechanisms
here. 

A first type of broadening, termed \emph{natural broadening}, occurs because the
time of the transition between the two energy levels is finite. The spectral
lines which have only been broadened by this type have a Lorentzian shape, i.e.,
have peaks of the form $1/(1+x^2/s^2)$ where $s$ is a width parameter. Usually
natural broadening leads to very narrow lines. Much larger than natural
broadening is usually \emph{Doppler broadening} which occurs because the
emitting (or absorbing) particles are in constant thermal motion which leads to
a Doppler effect which shifts the energies of the photons. The shape of spectral
lines which only have been Doppler broadened are Gaussian. While the width of
the Doppler broadened lines is proportional to the energy we will neglect this
here and assume a constant width approximation. Neighbouring particles to the
electrons emitting or absorbing the photons produce a third kind of broadening,
the \emph{pressure broadening}. One can show that in the case where only
pressure broadening occurs the spectral lines are Lorentzian. Further broadening
originates in the instrumentation and even discretisation (or binning) of the
spectrum produces a certain amount of broadening~\cite{DulK08}. Finally the
medium which the photons need to traverse before getting to the observer also
produces some broadening. There are other effects which contribute to broadening
and there are other distortions of spectra than broadening occurring. This
includes spectral shifts and the occurrence of extra peaks, so-called
satellites~\cite{Dra06}. 

A fairly general but simple broadening model would represent observed spectra as
the effect of an integral operator on an underlying spectrum which might have
been modified in other ways. Here this underlying spectrum $u$ is assumed to be
in $L_2(\R)$ and so an observed spectrum $g$ is of the form
$$
  g(x) = \int_\R a(x,y)\, u(y)\, dy
$$
with some kernel $a$ which in the simplest case is assumed to be a convolution
kernel, i.e. $a(x,y) = \alpha(x-y)$ for some $L_2$ function $\alpha$. More
generally, an observed spectrum is modeled as the image of a product of several
broadening operators $A_1,\ldots,A_n$, i.e., as $g=A_1\cdots A_n u$. In some
cases, such a product can lead to a normal distribution because of the central
limit theorem. Here we assume mostly that all the operators are convolutions and
have Lorentzian or Gaussian shape (but different widths). As the operators
commute and the convolution of Lorentzians is a Lorentzian and of Gaussians is a
Gaussian, respectively, it is found that a good model is given by the Voigt
shape which consists of a convolution of a Gaussian with a Lorentzian. In the
following we call the integral equation $Au=g$ representing any kind of (linear)
broadening the \emph{broadening equation}. 

\subsection{Fitting the lines}

While immediately appealing, the inversion of the broadening equation $Au=g$ is
not feasible as it is typically severely ill-posed, the $g$ has a substantial
amount of observational error and $u$ is typically not very smooth so that even
a regularised solution cannot be expected to be a good approximation. Any
feasible approximation will make use of the (approximate) Fraunhofer line
structure of the $u$. The simplest model assumes that $u$ is a measure with
discrete support and intensities $u_i$ so that the broadening equation takes the
form
$$
  g(x) = \sum_{i=1}^\infty a(x,x_i)\, u_i.
$$
The determination of the $x_i$ and $u_i$ from some data $g_\delta$ with
$\|g-g_\delta\| \leq \delta$ can be done by minimising the least-squares
objective function
$$
  J(u) = \left\|\sum_{i=1}^\infty a(\cdot,x_i)\, u_i - g_\delta \right\|.
$$
When the locations $x_i$ of the spectral lines are known this amounts to a
linear least squares problem. The determination of these locations, however, is
a nonlinear problem. An interesting discussion of this problem from the
perspective of Bayesian statistics can be found in~\cite{Bre88}. 

In~\cite{GolP73} Golub and Pereya discuss the \emph{variable projection method}
for the solution of the nonlinear problem above in the case of a finite number
of non-zero $u_i$. Rather than minimising the squared residual they first solve
for the linear parameters $u_i$ explicitly such that $\u =A^+(x)g_\delta$ where
$\u=(u_1,\ldots,u_n)$. They then use a nonlinear (typically Gauss-Newton) method
to solve for the locations $x=(x_1,\ldots,x_n)$ by minimising the functional
$\|A(x)A^+(x)g_\delta - g_\delta\|$. In a recent paper~\cite{MulS08} the authors
discuss the application of this method to spectroscopic problems and consider
reasons for the success of the approach. They observe in particular superior
numerical conditioning and convergence of the Gauss-Newton method compared to
the original optimisation problem.

An important condition required by the variable projection method is that the
matrix $A(x)$ has to have a fixed rank for $x$ in some neighbourhood of the
minimum of the variable projection functional. This condition may be difficult
to fulfil when one has two components of $x$ which are very close. As two
coinciding $x_i$ will reduce the rank of $A(x)$, the neighbourhood where the
rank condition holds can be very small. It would certainly be difficult to find
initial conditions for the Gauss-Newton iteration which are in a neighbourhood
of the exact solution.

When the spectral lines are well separated then the variable projection method
works very well. This is for example the case where the \emph{baseline
condition} in which case the functions $a(\cdot,x_i)$ have non-overlapping
supports (at least numerically). It follows that the $a(\cdot,x_i)$ are
pair-wise orthogonal, good starting values can be obtained and the rank
condition can be maintained. A similarly favourable situation occurs if the
\emph{Rayleigh condition holds}. This motivates the development of methods which
are able to enhance the spectrum so that the enhanced spectral lines are better
separated. A discussion of these aspects from a statistical perspective can be
found in~\cite{AcuH97}.

\section{Resolution enhancement}

\subsection{The enhancement equation\label{sec1.1}}
The resolution enhancement procedures considered here consist of algorithms
which determine the \emph{enhanced spectrum} $f$ as a solution of an integral
equation $Bf = g$ from the \emph{observed spectrum} $g_\delta$ which satisfies
$\|g_\delta-g\|\leq \delta$. The integral operator $B$ is of the form 
\begin{equation}
  Bf\,(x) = \int_\R b(x,y) f(y)\, dy.
\end{equation}
The integral equation
\begin{equation}
  \label{Bfeqg}
  Bf = g
\end{equation}
will be called the \emph{enhancement equation}. The operator $B$ is chosen such
that the enhanced spectrum $f$ has narrower lines than the original spectrum
$g$. The main constraint in choosing $B$ is that the enhancement equation should
be solvable which means that $g$ has to be in the range of $B$:
\begin{equation}
  \label{ginRB}
  g \in \range(B).
\end{equation}
In the case where $B$ is a convolution operator, the resolution enhancement is
the \emph{Stokes correction formula}~\cite{Sto48}. The integral equation Ansatz
for enhancement was introduced Allen, Gladney and Glarum in their
ground-breaking paper~\cite{AllGG64}. A simple precursor to this type of
enhancement is the \emph{Eddington correction formula}~\cite{Edd13,Edd40,Bra03}
for the enhancement of spectra with Gaussian peaks using differentiation.

The careful choice of the operator $B$ is essential to successful enhancement.
Even if the range condition~(\ref{ginRB}) holds, the solution of the enhancement
equation~(\ref{Bfeqg}) may show poor resolution and contain a large error. This
is due to the ill-posedness of the enhancement equation. It's solution will
require some form of regularisation. When selecting $B$ one has to trade-off the
amount of enhancement achievable by $B$ against the regularisation required for
the solution of the enhancement equation. While the theory of resolution
enhancement is based on the general theory for the solution of integral
equations, there is one important difference: When solving integral equations,
the operator is given while for resolution enhancement, the operator $B$ is
chosen. In both cases, on needs to choose the regularisation method.

There is a large literature on regularisors, a concise and short reference is
still the book by Groetsch~\cite{Gro84}. In this book, convergence rates of
regularisors are given, provided that a \emph{source condition} of the form $g
\in \range((BB^*)^s)$ holds for some integer $s>1$ and where $B^*$ denotes as
usual the adjoint of the operator $B$. Here we will use a more general theory
based on \emph{variable Hilbert scale inequalities}~\cite{Heg92,Heg95}. This
framework has since been used in~\cite{MatP03,MatT06,MatH08,MatS08}. In the
analysis literature, the variable Hilbert scale interpolation is called
interpolation with a function parameter\footnote{thanks to M.~Hansen and
S.~Kuehn for pointing this out to me}, see, for
example~\cite{Mer83,CobL86,MikM08}. In the analysis of partial differential
equations, a related generalised H\"older inequality has been applied
in~\cite{BegS07}. Source conditions are very important in the analysis of
convergence of regularisation and some newer work which includes the application
to nonlinear problems can be found in~\cite{DuvHY07,Tau08,MatH08,HeiH09}. The
recovery of $f=B^{-1} g$ from $g_\delta$ is the main topic of the
book~\cite{Gro07} by Groetsch. The specific case of singular convolutions are
covered in a paper by Sushkov~\cite{Sus96}.

In the following let $H_B\subset{L_2(\R)}$ denote the Hilbert space with the
norm $\|g\|_B = \|B^{-1} g\|$ and let $H_\psi$ be the Hilbert space with the
norm $\|g\|_\psi = (g,\psi((BB^*)^{-1})g)$ where $\psi$ is a function on (a
subset of) $(0,\infty)$ which is continuous and monotonically increasing. The
operator $\psi((BB^*)^{-1}))$ is defined using the spectral theorem as
in~\cite{Heg95}. In the following $f_\alpha$ will always denote a regularised
solution of $Bf=g$. One then has the following general convergence theorem. 
\begin{theorem}
  \label{thm1}
  Let $B: L_2(\R) \rightarrow L_2(\R)$ be an injective, continuous linear
operator for which $BB^*-\lambda I$ is injective for $0\leq\lambda\leq c_0$ for
some $c_0>0$. Furthermore, let $\psi$ be a non-negative function which is
monotonically increasing for arguments larger than $1/c_0$. Finally, let $\Psi$
be a non-negative function such that $\Psi(\psi(\lambda))\geq \lambda$ and $\Psi$
is monotonically increasing and concave for all arguments $\lambda > 1/c_0$.
  
  If $f_\alpha\in H_\psi$ satisfies
  \begin{align}
       \|Bf_\alpha\|_\psi  & \leq C, \quad \text{and} \label{stable}\\
       \|Bf_\alpha-g\|  & =   \epsilon \label{consistent}
  \end{align}
  then
  \begin{equation}
    \label{VHSineq}
    \|f-f_\alpha\| \leq \epsilon \,
    \sqrt{\Psi\left((C+\|g\|_\psi)^2/\epsilon^2\right)}
  \end{equation}
  for all $f$ and $g=Bf\in H_\psi$.
\end{theorem}
\begin{proof}
  The functions $\Psi$, $\psi$ together with the functions $\theta(\lambda)=1$ 
  and $\phi(\lambda)=\lambda$ satisfy the conditions of Theorem~1 
  in~\cite{HegA09} which is a direct consequence of the interpolation
  inequalities in~\cite{Heg92,Heg95} which, with $\|r\|_\phi = \|r\|_B$ and
  $\|r\|_\theta = \|r\|$ here takes the form
$$
  \|r\|_B \leq \|r\|\, \sqrt{\Psi(\|r\|_\psi^2/\|r\|^2)}, \quad
  \text{for all $r\in H_\psi$.}
$$

  Now let $r=Bf_\alpha -g $. As $g=Bf$ on has 
$$
   \|r\|_B=\|Bf_\alpha-Bf\|_B = \|f_\alpha - f\|.
$$
  Furthermore, by the triangle inequality one gets
$$
  \|r\|_\psi = \|Bf_\alpha - g\|_\psi \leq \|Bf_\alpha\|_\psi + \|g\|_\psi
$$
  and, as $\Psi$ is monotonically increasing, it follows that
$$
  \Psi(\|r\|_\psi^2/\|r\|^2) \leq \Psi((C+\|g\|_\psi)^2/\epsilon^2).
$$
  Inserting this in the interpolation inequality gives the claimed bound.
\end{proof}
This result can be interpreted as a variant of the Lax equivalence theorem. The
conditions on $f_\alpha$ are the stability condition $\|Bf_\alpha\|_\psi\leq C$
and the consistency condition $\|Bf_\alpha -g \|=\epsilon$. If $\psi$ is unknown
one may take a stronger norm for stabilisation in a discrepancy method similar
to the one discussed in~\cite{Heg92}. For consistency one wants to make sure
that $\epsilon$ is small. This is achieved indirectly by controlling the size of
$\|Bf_\alpha-g_\delta\|$ and observing that 
$$
  \|Bf_\alpha - g\| \leq \|Bf_\alpha - g_\delta\| + \|g-g_\delta\|
$$
by the triangle inequality. In the following we call any (approximate)
enhancement $f_\alpha$ which satisfies both conditions~(\ref{stable})
and~(\ref{consistent}) a \emph{spectrum which has been stably enhanced with
$B$}.

\subsection{The Eddington correction formula\label{sec1.2}}

This early and still popular approach to the enhancement of Gaussian spectra
uses derivatives and is of the form
$$
  f = g - \frac{g^{(2)}}{2} + \frac{g^{(4)}}{8} - \cdots,
$$
see~\cite{Edd13,Edd40,Bra03}. It has been observed in~\cite{AllGG64} that
correction formulas of this type may be viewed as solutions of integral
equations of the form discussed in section~\ref{sec1.1}. We can thus apply
theorem~\ref{thm1} to obtain an error bound for the Eddington correction. See
also~\cite{Mei77} for a discussion of their application in practice. Other
procedures to spectral enhancement based on differentiation are discussed from
the point of view of numerical differentiation in~\cite{AndH09}.

The $k$-th order Eddington correction $f$ is defined as
\begin{equation}
  \label{eqn:Eddi}
  f = \sum_{j=0}^k \frac{(-1)^j}{2^j j!} g^{(2j)}
\end{equation}
where $g^{(2j)}$ denotes the derivative of order $2j$ of $g$. The Eddington
correction formula now fits into the integral equation framework for resolution
enhancement with enhancement equation $Bf=g$ and the enhancement operator $B$
has a kernel
$$
  b(x,y) = \frac{1}{\pi} 
    \int_0^\infty \left(\sum_{j=0}^k \frac{\omega^{{2j}}}{2^j j!}
                 \right)^{-1} \cos(\omega(x-y))d\omega.
$$

In particular, for $k=1$ one has
$$
  b(x,y) = \frac{1}{\sqrt{2}} e^{-\sqrt{2}|x-y|}
$$
and for $k=2$ the kernel is of the form
$$
   b(x,y) = \gamma e^{-\alpha|x-y|} \cos(\beta(|x-y|+\theta)
$$
for some $\alpha, \beta, \gamma$ and $\theta$.

In the following, let
$$
  a_G(x,y) = \frac{1}{\sqrt{2\pi}} e^{-(x-y)^2/2}
$$
and let a spectrum $g$ which has been broadened by $a_G$ be called a
\emph{Gaussian spectrum}. In this case one has
\begin{equation}
  \label{Gauss-spec}
  g(x) = \int_\R a_G(x,y) u(y)\, dy
\end{equation}
for some $u\in L_2(\R)$. The Eddington correction formula have been designed
to reduce some of the broadening produced by $a_G$.

A motivation for this particular formula comes from the convolution theorem as
$$
  \hat{g}(\omega) = \hat{a}_G(\omega) \hat{u}(\omega)
$$
where $\hat{a}_G(\omega) = \exp(-\omega^2/2)$ and $\hat{g}$ and $\hat{u}$ are
the Fourier transforms of $g$ and $u$ respectively. By the Taylor theorem one 
then gets formally
$$
  \hat{u}(\omega) = \sum_{j=0}^\infty \frac{1}{2^j j!} \omega^{2j} 
  \hat{g}(\omega).
$$
Truncating this expansion and using the fact that multiplication with $\omega^2$
in the Fourier domain corresponds to taking $-d^2/dx^2$ in the original domain
gives the formula.

The following lemma provides the expressions and some properties for the $\psi$
and $\Psi$ which will be used to  establish the error bound of the correction
formula. 
\begin{lemma}
  \label{lem1}
  Let $B$ be the enhancement operator\footnote{Some times the inverse $B^{-1}$
  is called enhancement operator} for the $k$-th order Eddington
  correction formula. Furthermore, let $t_k(\eta)$ be the $k$-th order Taylor
  polynomial for the exponential function for $k\geq 0$ and $t_k=0$ for $k<0$.
  Then
  \begin{enumerate}
    \item Any Gaussian spectrum $g$ is in $H_\psi$, the Hilbert space with
     the scalar product $(g,g)_\psi = (g,\psi((BB^*)^{-1}) g)$ and where
     $$
       \psi(\lambda) = \exp(2 t_k^{-1}(\sqrt{\lambda})), \quad \lambda \geq 1.
     $$
    \item The inverse 
     $$
       \Psi(\eta) = \psi^{-1}(\eta) = t_k(\log(\eta)/2)^2, \quad \eta \geq 1
     $$
     is concave.
  \end{enumerate}
\end{lemma}
\begin{proof}
\begin{enumerate}
  \item The $B$-norm is by Parseval's theorem
$$
  \|g\|_B^2 = \|B^{-1} g\|^2 = \int_\R t_k(\omega^2/2)^2 |\hat{g}(\omega)|^2\,
             d\omega.
$$
As $\psi(t_k(\omega^2/2)^2) = \exp(\omega^2)$ by definition one gets
$$
  \|g\|_\psi = (g, \psi((BB^*)^{-1}) g) = \int_\R \exp(\omega^2) 
   |\hat{g}(\omega)|^2\, d\omega.
$$
which is equal to $\|u\|^2$ if $g$ is a Gaussian spectrum with 
$$
  g(x) = \int_\R a_G(x,y) u(y)\, dy.
$$
It follows that $\|g\|_\psi$ is a norm on the set of Gaussian spectra which
provides a Hilbert space structure for this space.
\item As $dt_k(\zeta)/d\zeta = t_{k-1}(\zeta)$ one has 
      $d\Psi(\zeta)/d\zeta = t_{k-1}(\zeta)$ and consequently
$$
  \frac{d^2\Psi}{d\zeta^2} = -\frac{1}{2\zeta^2}\left(\frac{1}{k!}
  \left(\frac{\log(\zeta)}{2}\right)^k t_{k-1} + \frac{1}{(k-1)!}
  \left(\frac{\log(\zeta)}{2}\right)^{k-1} t_k\right)
$$
which is non-positive and so $\Psi(\zeta)$ is concave for $\zeta\geq 1$.
\end{enumerate}
\end{proof}

We now get the main theorem which provides bounds on how well one can evaluate
the Eddingtion correction.

\begin{proposition}
  Let $f_\alpha$ be a stably enhanced spectrum using $B$ the $k$-th order
  Eddington enhancement for Gaussian spectra and 
  $\psi(\lambda) = \exp(2 t_k^{-1}(\sqrt{\lambda}))$. Then there exists a $C>0$
  independent of $\epsilon$ such that
\begin{equation}
  \|f - f_\alpha\| \leq C \epsilon |\log(\epsilon)|^k.
\end{equation}
\end{proposition}
\begin{proof}
  By theorem~\ref{thm1} and lemma~\ref{lem1} one has for 
  $1/\epsilon \geq C+\|g\|_\psi$:
\begin{align*}
  \|f-f_\alpha \| &\leq \epsilon \sqrt{\Psi((C+ \|g\|_\psi)^2/\epsilon^2)} \\
  &\leq \epsilon t_k(2\log(C+\|g\|_\psi) - 2 \log(\epsilon)) \\
  &\leq \epsilon e 2^k(\log(C+\|g\|_\psi)-\log(\epsilon))^k \\
  &\leq \epsilon e 4^k(-\log(\epsilon))^k \\
  &\leq C\epsilon |\log(\epsilon)|^k.
\end{align*}
as $t_k(\lambda) \leq e \lambda^k$ for $\lambda \geq 1$ 
\end{proof}
A consequence of this lemma is that the ill-posedness of the problem is really
an issue for very high derivatives only. However, it is necessary to use
regularisation nonetheless as otherwise the data errors would remove any
advantage of the resolution enhancement and typically render the so
"enhanced" spectrum useless. Allen et al.~\cite{AllGG64} provide similar
correction formulas to the Eddington formula for Lorentz spectra and also
provide other correction formulas determining the coefficients in different
ways, see also~\cite{HegA05}. The analysis of the accuracy of so enhanced 
spectra can be analysed in exactly the same way as the Eddington formula.

In order to compare the above error bound for the Eddington correction formula
with the ones which we will obtain for other enhancement methods, one could
restate it as
$$
   \|f - f_\alpha\| \leq C \epsilon^{\eta(\epsilon)}
$$
where the exponent is
$$
   \eta(\epsilon) = 1 - k \frac{\log|\log(\epsilon)|}{|\log(\epsilon)|}.
$$
The formula is valid asymptotically and we assume that $0<\epsilon\leq 1/e$.
It can be seen that the smallest exponent is now obtained for $\epsilon=e^{-e}$
as
$$
  \eta_{\min}  = 1-k/e
$$ 
and consequently
$$
   \|f - f_\alpha\| \leq C \epsilon^{1-k/e}.
$$
It follows that for $k=1,2$ one gets an error bound which is similar to the
one obtained for an enhancement obtained through sharpening, see~\cite{HegA09}.
One can also get similar bounds for larger $k$ a necessary condition on the
error in this case, however, is
$$
     \frac{\log|\log(\epsilon)|}{|\log(\epsilon)|} < 1/k
$$
and while first and second order Eddington corrections (with second and fourth
derivatives) should work well even in the case of larger errors, but for higher
order derivative corrections one does require smaller data errors.

\subsection{Stokes enhancement with a Gaussian kernel\label{sec1.3}}

By using Fourier transforms, Stokes~\cite{Sto48} was able to introduce more
general spectral correction formulas which amount to general deconvolutions. An
example of such a formula would use a Gaussian kernel of the form
\begin{equation}
 \label{Genh}
 b(x,y) = 
  \frac{1}{\sqrt{2\pi}\kappa} \exp\left(-\frac{(x-y)^2}{2\kappa^2}\right).
\end{equation}
One can see that a resolution enhancement using this kernel reduces the width of
a unit Gaussian spectral line from equation~(\ref{Gauss-spec}) from one to
$\sqrt{1-\kappa^2}$. The enhanced spectrum is again a Gaussian with no other
local maxima and no local minima. While such an approach can be generalised to
other than Gaussian spectra (see~\cite{HegA09}) it does require the knowledge
of the spectrum. As Gaussian spectral lines are very smooth, using this type
of enhancement for less smooth non-Gaussian spectra will lead to meaningless
results as the range condition is not satisfied in such a case. 

For the Gaussian case, however, one has the following result about the error
of a regularised enhancement $f_\alpha$:
\begin{proposition}
Let $g$ be a Gaussian spectrum which has been enhanced by an operator with
kernel $b$ given in equation~(\ref{Genh}). Then the stable approximation 
$f_\alpha$ satisfies the error bound:
\begin{equation}
  \|f-f_\alpha \| \leq C \epsilon^{1-\kappa^2}.
\end{equation}
\end{proposition}
\begin{proof}
  Using Fourier transforms and the Parseval equality one derives 
  $\psi(\lambda) = \lambda^{1/\kappa^2}$. As $\kappa\in(0,1)$ the inverse
  $\Psi(\eta) = \psi^{-1}(\eta) = \eta^{\kappa^2}$ is concave and the bound then
  follows from theorem~\ref{thm1}.
\end{proof}
Note that in this case the source condition is a of a classical form and thus
the error bound may also be obtained using methods from~\cite{Gro84}.

As the spectral enhancement reduces the width by a factor $\sqrt{1-\kappa^2}$ it
follows for example that a reduction of the width by a factor two is obtained
by solving an integral equation of the first kind with error $O(\epsilon^{1/4})$
if a stable method is used and $\epsilon$ is the data error.

\section{Enhancing Voigt spectra with unknown line shape}

While it is known that many spectra are of Voigt type, i.e., they contain a
mixture of Gaussian and Lorentz broadening it is often unknown,
 how much of both types
are current in any particular spectrum. We will now present an enhancement
procedure which utilises a Lorentz kernel for the enhancement of a Voigt
spectrum. 

The enhancement equation $Bf=g$ providing the enhancement is an integral equation
with a Lorentz kernel of the form
$$
  b(x,y) = \frac{1}{\kappa \pi} \frac{1}{1+(x-y)^2/\kappa^2}.
$$
Thus $Bf$ is again a convolution and the Fourier transform is
$$
  \hat{b}(\omega) = \exp(-\kappa|\omega|).
$$
The width parameter $\kappa$ has to be chosen similar to the width parameter
for the Gaussian sharpening discussed in section~\ref{sec1.3} or the order of
the Eddington correction formula of section~\ref{sec1.2}. In this choice one
considers the trade-off between the enhancement obtained through the narrower
lines in the spectra and the error from the solution of the integral equation.

Before discussing the general case of a Voigt spectrum we provide a bound for
the error of the Stokes correction with Lorentz kernel of a Gaussian spectrum.
\begin{lemma}
  \label{lem2}
  Let $B$ be the enhancement operator for the Stokes correction formula with a 
  Lorentz kernel with width $\kappa$. Then a Gaussian spectrum is in the space
  $H_\psi$ (based on $B$) with
$
  \psi(\lambda) = \exp((\log(\lambda)/(2\kappa))^2).
$
Furthermore, the inverse $\Psi(\eta)= \psi^{-1}(\eta) =
\exp(2\kappa\sqrt{\log(\eta)})$ is concave if $\kappa\leq \sqrt{2}$ or if $\eta
\geq \frac{\kappa}{2}+\sqrt{(\kappa/2)^2-1/2}.$
\end{lemma}
\begin{proof}
 We use the Fourier transforms of the kernel of $BB^*$ which is
$\exp(-2\kappa|\omega|)$ and $AA^*$ which is $\exp(-\omega^2)$ and the Parseval
equality to get $\psi$.

The second derivative of $\Psi$ is then
$$
  \frac{d^2\Psi}{d\eta^2} = -{\frac{\kappa\,\left(2\,\log \eta-2\,\sqrt{\log
\eta}\,\kappa+1\right)\, e^{2\,\sqrt{\log
\eta}\,\kappa}}{2\,\eta^2\,\left(\log \eta \right)^{{\frac{3}{2}}}}}
$$
and it follows that $\Psi$ is concave if $2\,\log
\eta-2\,\sqrt{\log\eta}\,\kappa+1\geq0$. The conditions then follow directly.
\end{proof}
It then remains to apply Theorem~\ref{thm1} to get the following error bound:
\begin{proposition}
  The error of a stably computed enhancement $f_\alpha$ of a Gaussian spectrum
  using the Stokes correction formula with a Lorentz kernel is bounded by
$$
  \|f_\alpha - f\| \leq \epsilon^{1-2\kappa/\sqrt{|\log\epsilon|}}
$$
for $0<\epsilon<\epsilon_0$ and some $\epsilon_0>0$.
\end{proposition}
\begin{proof}
  By theorem~\ref{thm1} and lemma~\ref{lem2} one has for $\epsilon>0$ and
  some $C$ which satisfy
$$
  \frac{1}{\epsilon} \geq C + \|g\|_\psi \geq \sqrt{\kappa/2+\sqrt{(\kappa/2)^2-1/2}}
$$
the bounds
\begin{align*}
  \|f-f_\alpha \| &\leq \epsilon \sqrt{\Psi((C+ \|g\|_\psi)^2/\epsilon^2)} \\
  &\leq \epsilon \exp(\kappa\sqrt{\log((C+\|g\|_\psi)^2/\epsilon^2)}) \\
  &\leq \epsilon \exp(2\kappa\sqrt{|\log\epsilon|} \\
  &\leq \epsilon^{1-2\kappa/\sqrt{|\log\epsilon|}}.
\end{align*}
\end{proof}
Note that here $\kappa$ is not the width of the enhanced spectrum but a 
parameter which controls how much enhancement is done. Thus a larger $\kappa$
corresponds to more enhancement and $\kappa=0$ to no enhancement. For example,
if one has $\epsilon \approx 10^{-3}$ and $\kappa = 0.7$ one gets an error
of approximately $O(\epsilon^{1/2})$.

As discussed in the introduction many spectra have undergone broadening both
with Gaussian and with Lorentz kernels. The resulting class of spectra are the
Voigt spectra. We assume here that we know that a given spectrum is in this 
class, however we do not assume that we know how much each of the two components
have contributed to the broadening. This is why we suggest a Stokes correction
with Lorentzian kernels.

Specifically, let the Lorentz kernel be
$$
  a_L(x,y) = \frac{1}{\sqrt{2}\pi}\, \frac{1}{1+\frac{(x-y)^2}{2}}
$$
with Fourier transform
$$
  \hat{a}_L(\omega) = \exp(-\sqrt{2}|\omega|).
$$
The Voigt spectrum (with mixing parameter $\theta$) is then defined by its 
Fourier transform
$$
  \hat{a}_V(\omega) = \hat{a}_G(\omega)^\theta \, \hat{a}_L(\omega)^{1-\theta},
$$
and the kernel is thus
$$
  a_V(x,y) = \frac{1}{2\pi} \int_\R e^{i\omega(x-y)} \hat{a}_V(\omega)\, d\omega.
$$
A Voigt spectrum is then of the form
$$
  g(x) = \int_\R a_V(x,y) u(y)\, dy
$$
for some $u\in L_2(\R)$ and $0<\theta \leq 1$. One then has
\begin{lemma}
  \label{lem3}
  Let $B$ be the enhancement operator for the Stokes correction formula with a 
  Lorentz kernel with width $\kappa$. Then a Voigt spectrum with parameter 
  $\theta$ is in the space
  $H_\psi$ (based on $B$) with
$
  \psi(\lambda) = \exp\left(\theta\left(\log(\lambda)/(2\kappa)\right)^2
        +\sqrt{8}(1-\theta)\log(\lambda)/(2\kappa)\right).
$
Furthermore, the inverse $\Psi=\psi^{-1}$ is defined by
$$
  \Psi(\eta) = \exp\left(\frac{2\kappa}{\theta}\left(
    \sqrt{2(1-\theta)^2+\theta\log(\eta)}-\sqrt{2}(1-\theta)\right)\right).
$$
and is concave if $\kappa \leq \sqrt{2}\,\theta^{3/2}$ or if $\eta\geq \eta_0$
for some $\eta_0>0$.
\end{lemma}
\begin{proof}
 We use the Fourier transforms of the kernel of $BB^*$ which is
$\exp(-2\kappa|\omega|)$ and $AA^*$ which is 
$\exp(-\theta\omega^2-\sqrt{8}(1-\theta)|\omega|)$ and the Parseval
equality to get $\psi$.

With $\zeta(\eta)=\sqrt{\log(\eta)+b}$, $a=2\kappa\theta^{-3/2}$ and
$b=2(1-\theta)^2\theta$ one then has for the second derivative of $\Psi$
$$
  \exp(\sqrt{2}(1-\theta))\frac{d^2\Psi}{d\eta^2} =
-\frac{ae^{a\zeta(\eta)}}{4\eta^2\zeta(\eta)^3}(2\zeta(\eta)^2-a\zeta(\eta)+1).
$$
One gets convexity for $\Psi$ if $2\zeta(\eta)^2-a\zeta(\eta)+1 \geq 0$ which
happens if $\kappa \leq \sqrt{2}\,\theta^{3/2}$, or for $\eta > \eta_0$ and
large enough $\eta_0$.
\end{proof}

Then an application of Theorem~\ref{thm1} provides again an error bound:
\begin{proposition}
  The error of a stably computed enhancement $f_\alpha$ of a Voigt spectrum
  with width parameter $\theta$ 
  using the Stokes correction formula with a Lorentz kernel is bounded by
$$
  \|f_\alpha - f\| \leq 
      \epsilon^{1-2\kappa/\sqrt{\theta|\log\epsilon|+(1-\theta)^2}}
$$
for $0<\epsilon<\epsilon_0$ and some $\epsilon_0>0$.
\end{proposition}
\begin{proof}
  By theorem~\ref{thm1}, lemma~\ref{lem3} and the monotonicity of $\Psi$
  one has for $\epsilon>0$ satisfying $\epsilon (C+\|g\|_\psi) \leq 1$
  the bounds
\begin{align*}
  \|f-f_\alpha \| &\leq \epsilon \sqrt{\Psi((C+ \|g\|_\psi)^2/\epsilon^2)} \\
  \|f-f_\alpha\| & \leq \epsilon \sqrt{\Psi(\epsilon^{-4})} \\
  &\leq \epsilon 
    \exp\left(\frac{\kappa}{\theta}
    \left(\sqrt{2(1-\theta)^2+4\theta|\log\epsilon|} - \sqrt{2}(1-\theta)\right)
    \right) \\
  & = \epsilon^{1-\eta(\epsilon)}
\end{align*}
where
\begin{align*}
  \eta(\epsilon) &= \frac{\sqrt{2}\kappa}{\theta}\, \left(
   \frac{\sqrt{(1-\theta)^2+2\theta|\log(\epsilon)|}-(1-\theta)}{|\log\epsilon|}
   \right)\\
  & =  \sqrt{2}\kappa
     \frac{2}{\sqrt{(1-\theta)^2+2\theta|\log\epsilon|}+(1-\theta)} \\
  &\leq \frac{2\kappa}{\sqrt{\theta|\log\epsilon|+(1-\theta)^2}}
\end{align*}
As $0 < \epsilon < 1$ an upper bound for $\eta(\epsilon)$ will lead to an upper
bound for the error.
\end{proof}

\section{Enhancing a Gaussian peak}

We provide some simple experiments which show how resolution enhancement 
modifies 
a single Gaussian peak. In Figure~\ref{fig:1} a Lorentz correction formula is
applied with different values of the parameter $\kappa$. Comparing the widths
at height 0.5 one sees that for $\kappa$ ranging from $\sqrt{2}$ to $4$ one
gets reductions of the widths between a factor of $1/2$ to almost $1/5$. 
Note that resolution enhancement comes at a cost which grows with $\kappa$ in
the sense that side bands start to occur. From the plot it appears that the
peaks of the side bands are at the level of the original (unenhanced) spectrum
but can be negative. 
In Figure~\ref{fig:2} one sees how regularisation (using the source condition)
does further distort the peak. In this case we choose $\kappa=2$. One can
clearly see the oscillations and the broadening which are caused by
regularisation.
Finally, Figure~\ref{fig:3} considers the same regularisation methods (except
for the case of the regularisation parameter $\alpha=0$ which gives much larger
errors). Here a data error of 5\% has been included. The effect of the error
onto the enhanced signals is that mainly the oscillations away from the centre
are strongly affected by data error, especially for the case of a small 
regularisation parameter.
\begin{figure}[ht]
	\centering
	\includegraphics[width=\textwidth]{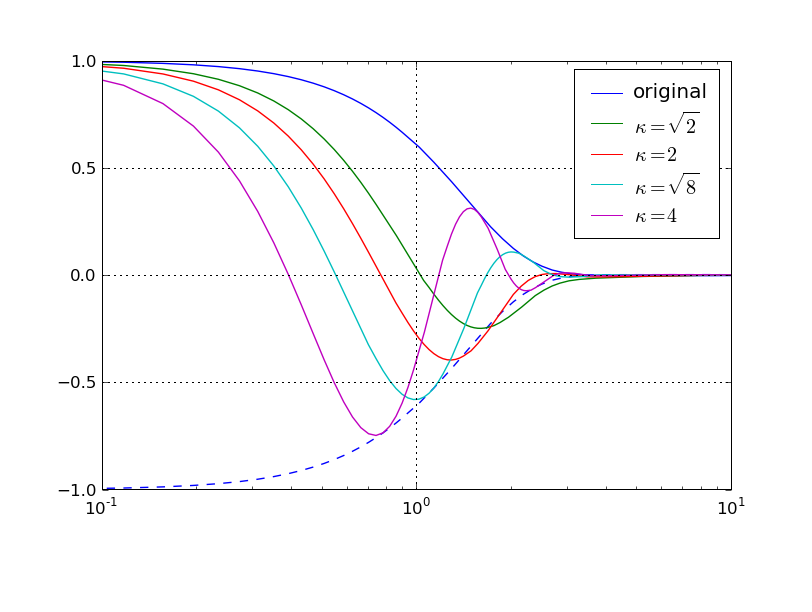}
	\caption{Lorentzian correction of a Gaussian}
	\label{fig:1}
\end{figure}
\begin{figure}[ht]
	\centering
	\includegraphics[width=\textwidth]{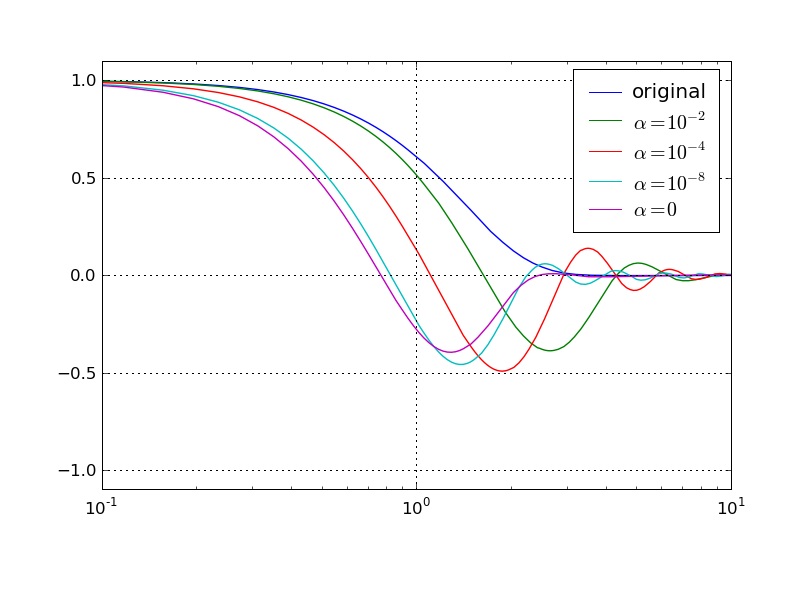}
	\caption{Regularised Lorentzian correction of a Gaussian}
	\label{fig:2}
\end{figure}
\begin{figure}[ht]
	\centering
	\includegraphics[width=1.00\textwidth]{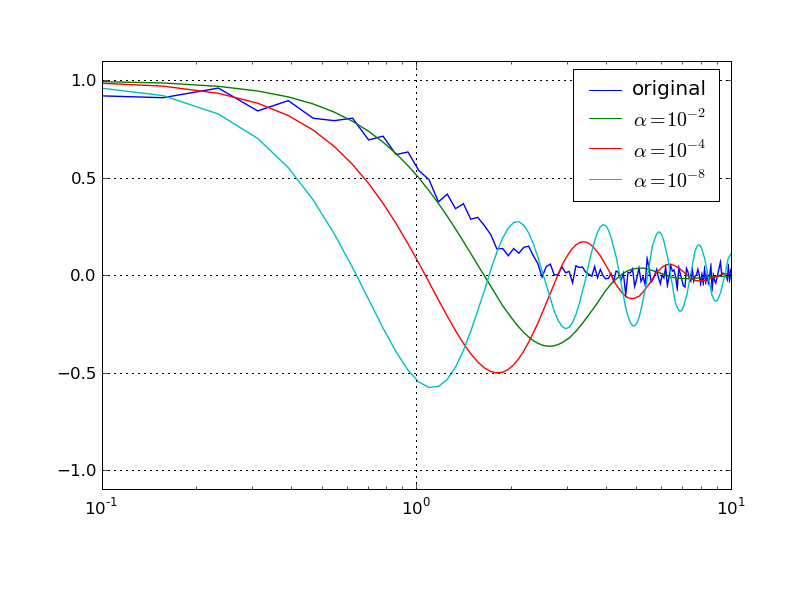}
	\caption{Regularised Lorentzian correction of a Gaussian with
           5\% data error}
	\label{fig:3}
\end{figure}

\section{Conclusion}

From the physics of spectral broadening one obtains the broadening equation $Au
= g$. For various reasons including the severe ill-posedness of the equations,
the fact that $A$ might not be known, and that $u$ might not be sufficiently
smooth, the solution of $Au=g$ is typically not feasible. However, this equation
provides a regularity or source condition for the solution of enhancement
equations $Bf = g$ which are essential for obtaining error bounds or
regularisation methods. As typically $A^*A$ is not a power of $B^*B$ the
standard convergence theory for ill-posed problems cannot be used. Instead we
apply the variable Hilbert scale theory and obtain convergence results for
Eddington correction and Lorentz deconvolution of Gaussian and Voigt spectra in
particular. Knowing these error bounds provides some insight into the choice of
the enhancement operators $B$ which goes beyond the range condition $\range(A)
\subset \range(B)$.

By a change of perspective one interprets resolution enhancement as an
application of an unbounded operator $R$. In the case of this paper, $R=B^{-1}$
for the integral operator $B$. Another larger class of such enhancements is
obtained when $R$ is a differential operator. The theory of the application of
such operators is covered in the recent book~\cite{Gro07} by Groetsch. A
specific algorithm for numerical differentiation based on averaging and
differences which converges with the size of the sampling with is analysed
in~\cite{AndHH98}. The important question of the choice of the amount of
differentiation for enhancement is discussed in~\cite{AndH09}.

If the broadening operator is known explicitly and is a convolution a different
approach to resolution enhancement is based on the dilation (or rather
contraction) of the spectral lines. Error bounds can also be obtained and a
variant of variable Hilbert scales, the \emph{dilational Hilbert scales} has
been introduced to perform this analysis in~\cite{HegA09}. The approach has a
particular appeal in practice as it does not introduce any satellite maxima. Such
maxima might still occur, however, when data errors are large and regularisation
has to be used.

There is a substantial practical literature on separating overlapping
line-shapes which cannot be covered here in any detail. As an example of a
method which uses extra information, i.e., the ratio of the heights of two lines
and the distance between them is the Rachinger correction formula~\cite{Rac48}.
This formula allows the determination of the corresponding line strengths $u_i$
even without knowledge of the shapes $a(\cdot,x_i)$. In a sense, this is also
what spectral enhancement methods attempt to achieve -- but without any extra
information. 

Related to the problem of spectral enhancement is the statistical problem of
deconvolution of a density. Convergence rates have been found for several such
problems in~\cite{CarH88}. These problems are often severely ill-posed and very
slow convergence rates are obtained. The reason for this is that one can only
assume that the underlying density is $k$ times differentiable. While the
authors did not use spectral theory nor the variable Hilbert scale interpolation
inequality for their results one can obtain similar results with these more
modern tools. This work has been continued and practical estimators are
discussed (also for less severely ill-posed problems) in~\cite{Fan91}. An
interesting adaptive approach to these statistical problems is discussed
in~\cite{Tsy00} where similar convergence results are obtained as in our
discussion but using different techniques for analysis and different algorithms,
see also~\cite{Mei08,JonM09}. It would certainly be of interest to investigate
these approaches from an ill-posed problem perspective using variable Hilbert
scales.

Maybe the most important limitation of the above discussion relates to the fact
that all the operators occurring are convolutions. As outlined in the discussion
of the models on broadening, the Doppler broadening is not a convolution and
one can see that the operator may be factorised into diagonal operators and 
a convolution. The next natural step would be to utilise norm
equivalences (possibly using wavelets) with the variable Hilbert scale 
interpolation theory to deal with such more general source conditions.

\section*{Acknowledgements}
This paper is dedicated to Chuck Groetsch who made an important contribution
to the original variable Hilbert scale paper~\cite{Heg92} by encouraging the
author at a crucial point in the process.

\bibliography{manu}
\bibliographystyle{amsplain}

\end{document}